\documentclass[11pt]{article}
\usepackage{amsmath, amssymb, amsfonts, amsthm, mathrsfs, relsize, mathtools,graphicx}
\usepackage[labelsep=period, font=footnotesize, labelfont=bf]{caption}

\newtheorem{theorem}{Theorem}[section]
\newtheorem{lemma}[theorem]{Lemma}

\newtheorem{cor}[theorem]{Corollary}

\newtheorem{thm}[theorem]{Theorem}
\newtheorem{predefinition}[theorem]{{\bf Definition}}
\newenvironment{definition}{\begin{predefinition}\rm{\hspace{-0.5 em}{\bf}}}{\end{predefinition}}
\newtheorem{prequestion}[theorem]{{\bf Question}}

\DeclareMathAlphabet{\mathbbmsl}{U}{bbm}{m}{sl}

\DeclareMathAlphabet{\mathpzc}{OT1}{pzc}{m}{it}
\DeclarePairedDelimiter\floor{\lfloor}{\rfloor}
\DeclarePairedDelimiter\ceil{\lceil}{\rceil}

\title{Percolating sets in  bootstrap percolation   on  the Hamming graphs}

\author{
M.R.  Bidgoli$^{^1}$  \quad  A. Mohammadian$^{^2}$  \quad   B. Tayfeh-Rezaie$^{^1}$\\[.3cm]
$^{^1}$School of Mathematics,\\
Institute for Research in Fundamental Sciences\,(IPM),\\
P.O. Box 19395-5746, Tehran, Iran\\[.1cm]
$^{^2}$School of Mathematical Sciences, Anhui University,\\
Hefei 230601, Anhui,  China\\[.2cm]
\textsf{\{bd,  ali\_m,  tayfeh-r\}}@\textsf{ipm.ir}}

\date{}

\begin{document}

\maketitle

\begin{abstract}
\noindent For any    integer $r\geqslant0$, the $r$-neighbor bootstrap percolation on a graph     is an  activation process of the  vertices. The process starts with   some   initially activated vertices  and  then, in each round,    any inactive vertex with at least $r$ active neighbors becomes activated.
A  set of  initially activated vertices  leading   to the activation of all vertices     is said to be a  percolating set.
Denote the  minimum  size of a  percolating set in the  $r$-neighbor bootstrap percolation process  on a graph $G$  by $m(G, r)$.  In this paper, we present upper and  lower bounds
on   $m(K_n^d, r)$, where $K_n^d$  is   the Cartesian product of $d$ copies of the complete graph $K_n$ which is referred   as     the  Hamming graph.
Among other results, we show  that $m(K_n^d, r)=\frac{1+o(1)}{(d+1)!}r^d$ when  both  $r$ and $d$  go to infinity with $r<n$ and    $d=o(\!\sqrt{r})$. \\


\noindent {\bf AMS Mathematics Subject Classification\,(2010):}    05C35, 60K35.
\end{abstract}

\section{Introduction}

Bootstrap percolation process on graphs    can be interpreted as  a  cellular automaton,    a concept  was  introduced    by von Neumann \cite{von}.
It   has been extensively investigated    in several diverse fields such as  combinatorics,  probability theory, statistical physics  and social sciences.
The    $r$-neighbor model  is the most studied version  of  this process in the literature. It  was  introduced  in 1979  by Chalupa, Leith  and Reich \cite{cha}.
In  the    $r$-neighbor bootstrap percolation  process on a graph,  first some vertices are initially activated and then,  in each phase, any  inactive vertex with at least $r$ active neighbors becomes activated.   Once a vertex becomes activated, it remains  active forever.
This  process has  also been treated  in the literature under  other    names like  irreversible threshold, influence propagation   and dynamic monopoly.

Throughout   this  paper, all graphs    are assumed to be finite,    undirected,     without loops and  multiple edges.
For   a graph $G$, we denote  the vertex set and the edge set of   $G$  by $V(G)$ and $E(G)$, respectively. For a vertex     $v$ of $G$,  we set   $N(v)=\{x\in V(G) \, | \, x \text{   is adjacent to } v\}$.
The {\sl degree} of $v$ is defined  to be  $|N(v)|$.
Given a nonnegative    integer $r$ and a graph $G$, the  $r$-{\sl neighbor bootstrap percolation process} on $G$  begins with a subset  $A_0$ of $V(G)$ whose elements are     initially activated   and then,   at    step $i$ of the process,    the set $A_i$ of active  vertices    is $$A_i=A_{i-1}\cup\Big\{v\in V(G) \, \Big| \,  |N(v)\cap A_{i-1}|\geqslant r\Big\}$$ for any   $i\geqslant1$.
We say   $A_0$    is a {\sl percolating set} of $G$   if $\bigcup_{i\geqslant0} A_i=V(G)$.
The main extremal problem here  is to determine the minimum size of a percolating set which  is denoted by $m(G,r)$.
The size of percolating sets has been studied for      various families of  graphs such as  hypercubes \cite{mor},   grids \cite{pete,ham},  tori \cite{ham},     trees \cite{rie} and random graphs \cite{fei,janson}.

Let us fix some notation and terminology. The {\sl Cartesian product} of two graphs $G$ and $H$, denoted by $G\square H$, is the graph with vertex set $V(G)\times V(H)$  in which two vertices $(g_1, h_1)$ and $(g_2, h_2)$ are adjacent if and only if either $g_1=g_2$ and $h_1$ is adjacent to $h_2$  or $h_1=h_2$ and $g_1$ is adjacent to $g_2$. We denote the    complete graph on $n$ vertices by  $K_n$ and we  consider   $[\![n]\!]=\{0, 1, \ldots, n-1\}$ as the vertex set of $K_n$.
Denote by  $K_n^d$     the  Cartesian product of $d$ vertex disjoint  copies of $K_n$, that  is, the   Hamming graph of dimension $d$.

In this paper, we present upper and  lower bounds on   $m(K_n^d, r)$. In particular, we establish  that $m(K_n^d, r)=\frac{1+o(1)}{(d+1)!}r^d$ when  both  $r$ and $d$  go to infinity with $r<n$ and    $d=o(\!\sqrt{r})$.  It is worth to mention  that   a random version of    the   $r$-neighbor bootstrap percolation  process on the   Hamming graphs has been investigated  in \cite{gra}.

\section{Two-dimensional Hamming   graphs}\label{sec2}

For every  integers  $n\geqslant1$ and $r\geqslant0$, it is clear  that $m(K_n, r)=\min\{n, r\}$.
In this section, we deal with the first nontrivial case, that is,  the Hamming graph of dimension $2$. We     derive an exact formula for $m(K_n^2,r)$. If $n\leqslant\ceil{r/2}$, then the degree of  any vertex of  $K_n^2$ is at most  $r-1$,  implying that   $m(K_n^2, r)=n^2$. The following theorem resolves the remaining cases.

\begin{theorem}\label{dim2}
For every  nonnegative  integers $n$ and $r$ with   $n\geqslant\ceil{r/2}+1$,
$$m\!\left(K_n^2, r\right)=\left\lfloor\frac{(r+1)^2}{4}\right\rfloor.$$
\end{theorem}

\begin{proof}
Let $$\mathcal{V}_{n, r}=\left\{(x, y)\in[\![n]\!]^2  \,    \left| \, x+(n-1-y)<\left\lceil\frac{r}{2}\right\rceil  \,    \text{ or }  \,   (n-1-x)+y<\left\lfloor\frac{r}{2}\right\rfloor\right.\right\}.$$
As  an  example,   $\mathcal{V}_{6, 5}$ is shown in Figure  \ref{f1f}.
It is well known that the number of  solutions of   $x_1+\cdots+x_k<m$ for the  nonnegative integers $x_1, \ldots, x_k$  is ${{m+k-1}\choose{k}}$. As $n\geqslant\ceil{r/2}+1$, we have
$$|\mathcal{V}_{n, r}|={{\left\lceil\tfrac{r}{2}\right\rceil+1}\choose{2}}+{{\left\lfloor\tfrac{r}{2}\right\rfloor+1}\choose{2}}=\left\lfloor\frac{(r+1)^2}{4}\right\rfloor.$$
Note that $\mathcal{V}_{n, r}\cap[\![n-1]\!]^2=\mathcal{V}_{n-1, r-2}$.
We  prove   by induction on $r$ that  $\mathcal{V}_{n, r}$  is a percolating set  in the $r$-neighbor bootstrap percolation  process  on $K_n^2$.   The statement is trivial for $r=0, 1$.
Let  $r\geqslant2$ and assume that   the vertices in $\mathcal{V}_{n, r}$ are  initially activated. The points on the lines $x=n-1$ and $y=n-1$  become activated   from top to bottom and from  right to left,   respectively.  Remove from $K_n^2$ all the vertices in the set
$$L=\left\{(x,y)\in[\![n]\!]^2 \,  \Big| \, x=n-1 \,    \text{ or } \,  y=n-1\right\}$$ to get $K_{n-1}^2$.
By   the induction hypothesis,   $\mathcal{V}_{n-1, r-2}=\mathcal{V}_{n, r}\cap[\![n-1]\!]^2$ is a percolating  set of  $K_{n-1}^2$   in the $(r-2)$-neighbor bootstrap percolation process.
Since each vertex in  $[\![n-1]\!]^2$  has two additional activated  neighbors in $L$,  we conclude  that $\mathcal{V}_{n-1, r-2}\cup L$ is a percolating  set of  $K_n^2$
in the $r$-neighbor bootstrap percolation process. This proves the assertion.

We next use induction on $r$ to establish  that any  percolating set of     $K_n^2$ in the $r$-neighbor bootstrap percolation  process  has  at least
$\floor{(r+1)^2/4}$ elements.     The statement is trivially true  for $r=0, 1$.
Let  $r\geqslant2$ and consider a percolating set $A$  in the  $r$-neighbor bootstrap percolation  process on $K_n^2$.
Without   loss of generality,  one may assume that $(n-1, n-1)$  is the first vertex  in $[\![n]\!]^2\setminus A$  that becomes activated. So,    $(n-1, n-1)$  must have  at least $r$ initially activated neighbors in $L$, meaning that $|A\cap L|\geqslant r$. Remove  from $K_n^2$ all vertices  in $L$ to get  $K_{n-1}^2$. Since  $A\cup L$ is a percolating set  in the  $r$-neighbor bootstrap percolation  process  on $K_n^2$ and each vertex  in $[\![n-1]\!]^2$ has exactly two neighbors in $L$, we deduce  that $A\cap[\![n-1]\!]^2$ is a percolating set   of  $K_{n-1}^2$ in the $(r-2)$-neighbor bootstrap percolation  process. It follows from  the induction hypothesis
that $|A\cap[\![n-1]\!]^2|\geqslant\floor{(r-1)^2/4}$. Therefore,
\begin{equation*}|A|\geqslant|A\cap L|+\Big|A\cap[\![n-1]\!]^2\Big|\geqslant r+\left\lfloor\frac{(r-1)^2}{4}\right\rfloor=\left\lfloor\frac{(r+1)^2}{4}\right\rfloor.\qedhere\end{equation*}
\end{proof}
\begin{figure}[h]
\centering
\includegraphics[width=0.3\textwidth]{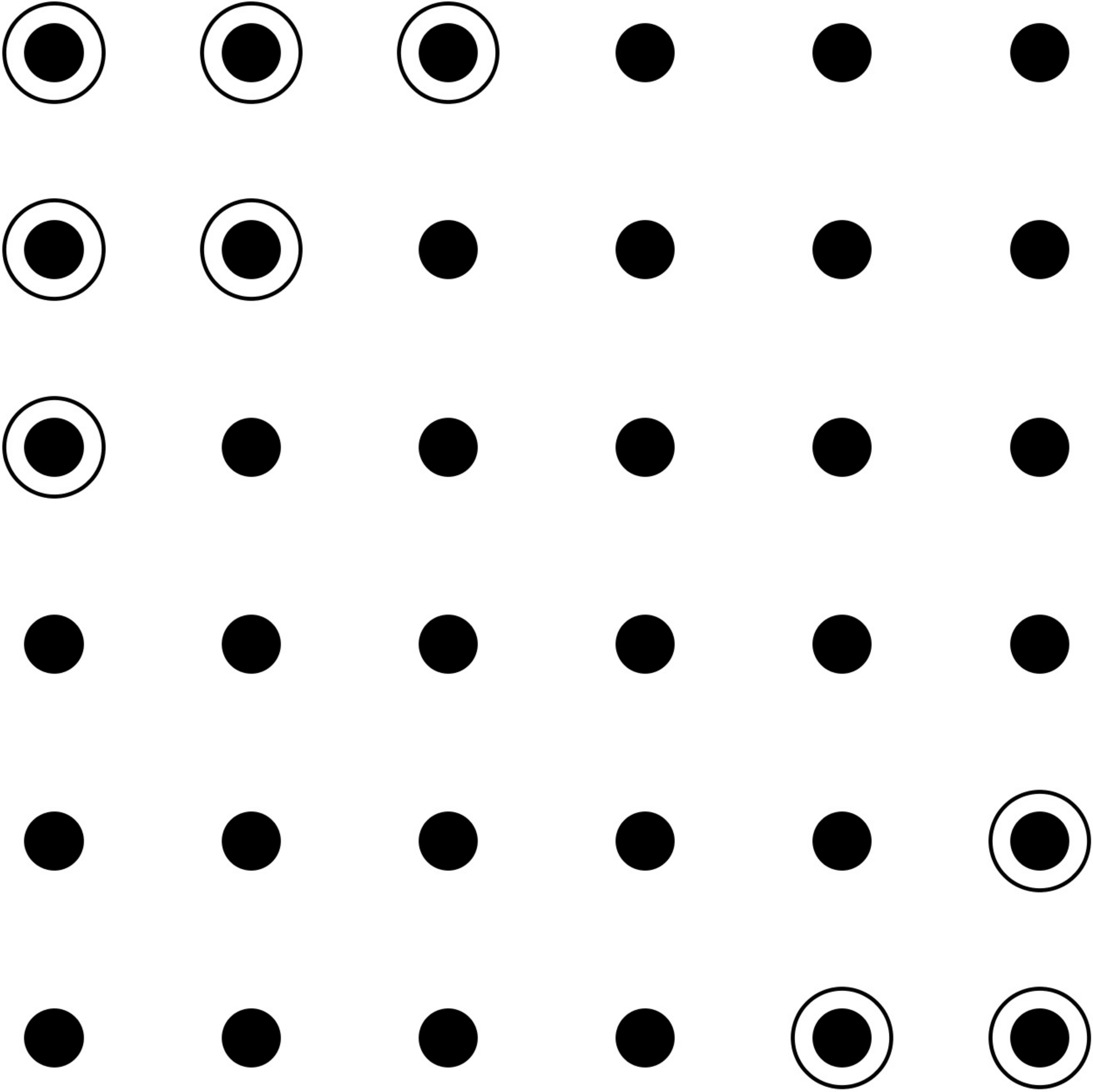}
\caption{The set $\mathcal{V}_{6, 5}$ is  outlined with   circles  drawn around  its elements.}\label{f1f}
\end{figure}

%\section{Star bootstrap percolation}

\section{Polynomial method}

Closely related to the $r$-neighbor bootstrap percolation    is the notion of graph bootstrap percolation which was introduced by Bollob\'as in 1968 under the name of `weak saturation' \cite{bol} and was later studied
in 2012 by  Balogh,     Bollob\'as     and   Morris \cite{bal}. We recall  the formal definition.
Given two  graphs $G$ and $H$, the  $H$-{\sl bootstrap percolation process} on $G$  begins with a   subset  $E_0$ of $E(G)$ whose elements are initially activated   and then,   at   step $i$ of the process,    the set  of activated edges   is
$$E_i=E_{i-1}\cup\left\{
e\in E(G) \, \left| \,
\begin{array}{ll}
    \text{There exists  a subgraph $H_e$ of $G$ such} \\
    \text{that $H_e$ is isomorphic to $H$, $e\in E(H_e)$} \\
    \text{and $E(H_e)\setminus\{e\}\subseteq E_{i-1}$.}
   \end{array}
    \right.
\right\}$$
for any   $i\geqslant1$.
The  set $E_0$  is called    a {\sl percolating set} of  $G$  provided   $\bigcup_{i\geqslant0}E_i=E(G)$.  The minimum size of a percolating  set in the $H$-bootstrap percolation  process  on $G$ is said to be  the {\sl weak saturation number} of $H$ in $G$ and is denoted by $\text{\sl wsat}(G, H)$. For simplicity  and following  \cite{ham}, we let  $m_e(G, r)=\text{\sl wsat}(G,S_{r+1})$,  where $S_{r+1}$ is the star graph   on  $r+2$ vertices. It is easy to verify   that $m_e(G,r)\leqslant rm(G,r)$.   Using this inequality and  by  computing $m_e(K_n^d, r)$ in the current  section,  we will present    a lower bound on $m(K_n^d, r)$  in the next  section.  We first recall the following definition from  \cite{ham}.

\begin{definition}\label{defhh}
Let $r$ be a positive  integer and   let   $G$  be a graph equipped  with a  proper edge coloring   $c: E(G)\longrightarrow\mathbbmsl{R}$. Let $W_c(G, r)$ be the vector space over $\mathbbmsl{R}$ consisting of all
functions $\phi : E(G)\longrightarrow\mathbbmsl{R}$ for   which there exist  polynomials $\{P_v(x)\}_{v\in V(G)}$ satisfying
\begin{itemize}
\item[{\rm (i)}] $\deg P_v(x)\leqslant r-1$ for any vertex  $v\in V(G)$;
\item[{\rm (ii)}]  $P_u(c(uv))=P_v(c(uv))=\phi(uv)$ for each edge $uv\in E(G)$.
\end{itemize}
It is said that the polynomials  $\big\{P_v(x)\big\}_{v\in V(G)}$ {\sl recognize} $\phi$.
\end{definition}

The following theorem  provides an interesting  linear algebraic lower bound on $m_e(G,r)$.

\begin{theorem}[Hambardzumyan,   Hatami,     Qian \cite{ham}]\label{hh}
Let  $r$ be a positive  integer  and let $c : E(G) \longrightarrow\mathbbmsl{R}$ be a proper edge coloring of a graph $G$. Then $m_e(G, r)\geqslant\dim W_c(G, r)$.
\end{theorem}

\begin{lemma}\label{lmekn}
For every positive    integers $n$ and $r$ with $n\geqslant r+1$, there exists a proper edge  coloring $c: E(K_n)\longrightarrow\mathbbmsl{R}$ such that $\dim W_c(K_n, r)\geqslant{{r+1}\choose{2}}$.
\end{lemma}

\begin{proof}
We introduce an edge coloring $c$ and $\binom{r+1}{2}$ independent vectors in $W_c(K_n,r)$.
Fix  arbitrary distinct nonzero real numbers $\gamma_0, \gamma_1, \ldots, \gamma_{n-1}$ and let   $c(ij)=\gamma_i\gamma_j$ for any   edge $ij\in E(K_n)$. Obviously, $c: E(K_n)\longrightarrow\mathbbmsl{R}$  is a proper edge  coloring of $K_n$.
For each   edge $uv\in E(K_n)$ with  $u, v\in[\![r+1]\!]$,   we define   polynomials $P^{uv}_0(x), P^{uv}_1(x), \ldots, P^{uv}_{n-1}(x)$  as follows.  For any  $i\in[\![n]\!]$, let
$$P^{uv}_i(x)=\left\{\begin{array}{ll}
 \quad 0,  &   \mbox{if } i\in[\![r+1]\!]\setminus\{u, v\};\\
 \vspace{0.5mm}\\
{\mathlarger{\mathlarger{\prod_{{k\in[\![r+1]\!]}\atop{k\notin\{u, v\}}}}} \, {\mathlarger{\frac{x-\gamma_i\gamma_k}{\gamma_u\gamma_v-\gamma_i\gamma_k}}}},  &   \mbox{if } i\in\{u, v\};\\
\vspace{0.5mm}\\
{\mathlarger{\mathlarger{\prod_{{k\in[\![r+1]\!]}\atop{k\notin\{u, v\}}}}} \, {\mathlarger{\frac{(x-\gamma_i\gamma_k)(\gamma_i-\gamma_k)}{\gamma_i(\gamma_u-\gamma_k)(\gamma_v-\gamma_k)}}}},  &  \mbox{if } i\in\{r+1, \ldots, n-1\}.\end{array}\right.$$
It is not hard to check that $\deg(P^{uv}_i)\leqslant r-1$ and $P^{uv}_i(c(ij))=P^{uv}_j(c(ij))$.
Define  $\phi_{uv}: E(K_n)\longrightarrow\mathbbmsl{R}$ as $\phi_{uv}(ij)=P^{uv}_i(c(ij))$. Note that  $\phi_{uv}$  vanishes    on each  edge $ij$ with    $i, j\in[\![r+1]\!]$ except on $uv$.  From this,  it follows that  $\{\phi_{uv}\}_{u, v\in[\![r+1]\!]}$  is a  linearly  independent subset of   $W_c(K_n, r)$. This completes  the proof.
\end{proof}

\begin{lemma}\label{car}
Let   $n, r$ be  two   positive integers  and let  $c : E(G)\longrightarrow\mathbbmsl{R}$ be a proper edge coloring of a graph $G$.  Then, there is a proper edge  coloring $\widehat{c} : E(G\square K_n)\longrightarrow\mathbbmsl{R}$ such that $$\dim W_{\widehat{c}}\,(G\square K_n, r)\geqslant\sum_{t=0}^{n-1}\dim W_c(G, r-t),$$
where $W_c(G, i)$ is defined  to be  $\{0\}$ if  $i\leqslant0$.
\end{lemma}

\begin{proof}
Consider arbitrary distinct nonzero real numbers $\gamma_0, \gamma_1, \ldots, \gamma_{n-1}$  such that none of the numbers   $\gamma_i\gamma_j$ is   in the image of $c$. For every   two adjacent vertices $u=(g, i)$
and $v=(h, j)$ of  $G\square K_n$,  define
$$\widehat{c}(uv)=\left\{\begin{array}{ll}
c(gh),  &   \mbox{if } i=j;\\
\vspace{-1mm}\\
\gamma_i\gamma_j,  &  \mbox{if } g=h.\end{array}\right.$$
Fix  $t\in[\![n]\!]$,  a basis  $\mathscr{B}_t$ for  $W_c(G, r-t)$   and  a function   $\phi\in\mathscr{B}_t$.  According to Definition \ref{defhh}, there exist polynomials $\{P^{\phi}_g(x)\}_{g\in V(G)}$
recognizing  $\phi$. Define    polynomial $Q^{t,\phi}_u$ for any vertex $u=(g,i)\in V(G\square K_n)$
as   $Q^{t,\phi}_u(x)=P^{\phi}_g(x)\mathnormal{\Gamma}^t_i(x),$ where
$$\mathnormal{\Gamma}^t_i(x)=\prod_{\ell=0}^{t-1}(\gamma_i-\gamma_\ell)\left(\frac{x}{\gamma_i}-\gamma_\ell\right).$$
Note that $\mathnormal{\Gamma}^t_i(\gamma_i\gamma_j)=\mathnormal{\Gamma}^t_j(\gamma_i\gamma_j)$ for all   $i$ and $j$. Also,    we know from Definition \ref{defhh} that  $P^{\phi}_g(c(gh))=P^{\phi}_h(c(gh))$ for each edge $gh\in E(G)$. Hence,     $Q^{t,\phi}_u$ and $Q^{t,\phi}_v$ have the same value on $\widehat{c}(uv)$  for any  edge  $uv\in E(G\square K_n)$.
This implies that  $\{Q^{t,\phi}_u\}_{u\in V(G\square K_n)}$   recognize   a function $\mathnormal{\Psi}_{t,\phi}\in W_{\widehat{c}}(G\square K_n, r)$.

Since we may choose  the pair  $(t, \phi)$ in $\sum_{t=0}^{n-1}\dim W_c(G, r-t)$ different ways,  it remains  to show that all functions $\mathnormal{\Psi}_{t,\phi}$ are linearly independent.
Suppose  that $\sum_{t,\phi}\lambda_{t,\phi}\mathnormal{\Psi}_{t,\phi}=0$ for some scalars   $\lambda_{t,\phi}\in\mathbbmsl{R}$. Towards  a contradiction, assume   that $\tau$ is the smallest value such that $\lambda_{\tau,\phi}\neq0$ for some $\phi$. Obviously,  $\mathnormal{\Gamma}^t_i=0$ for any   $i<t$. This yields  that   $Q^{t,\phi}_{(g, \tau)}=0$ for every integer   $t>\tau$ and   vertex $g\in V(G)$.  Thus, for every two adjacent vertices   $u=(g, \tau)$ and $v=(h, \tau)$ in $G\square K_n$, we have   $$\sum_{t,\phi}\lambda_{t,\phi}\mathnormal{\Psi}_{t,\phi}(uv)=\sum_{t,\phi}\lambda_{t,\phi}Q^{t,\phi}_u\big(\widehat{c}(uv)\big)=
\sum_{\phi\in \mathscr{B}_\tau}\lambda_{\tau,\phi}P^{\phi}_g\big(c(gh)\big)\mathnormal{\Gamma}^\tau_\tau\big(c(gh)\big)=0.$$ Our assumption on  $\gamma_0, \gamma_1, \ldots, \gamma_{n-1}$   implies that  $\mathnormal{\Gamma}^\tau_\tau(c(gh))\neq0$. Therefore,
$$\left(\sum_{\phi\in \mathscr{B}_\tau}\lambda_{\tau,\phi}\phi\right)(gh)=\sum_{\phi\in \mathscr{B}_\tau}\lambda_{\tau,\phi}P^{\phi}_g\big(c(gh)\big)=0$$ for each edge  $gh\in E(G)$. This is a contradiction, since $\mathscr{B}_\tau$ is a basis for  $W_c(G, r-\tau)$.
\end{proof}

\begin{lemma}\label{umeg}
Let $n, r$ be  two   positive integers  and let  $G$ be a   graph  all whose vertices are of degree   at least $r$.
Then $$m_e(G\square K_n, r)\leqslant\sum_{t=0}^{n-1} m_e(G, r-t),$$
where $m_e(G, i)$  is defined  to be  $0$ if  $i\leqslant0$.
\end{lemma}

\begin{proof}
For any $t$ with  $0 \leqslant t\leqslant \min\{r, n-1\}$, consider the    subgraph $G_t$  of $G\square K_n$ induced
by  $\{(v, t)\in V(G\square K_n) \,     | \,    v\in V(G)\}$ which  is clearly  isomorphic to $G$.  Also, consider   a  percolating  set $U_t$  of the  minimum possible size in the $S_{r-t+1}$-bootstrap percolation  process  on $G_t$  and activate its elements. We show that the edges of $G_0,   \dots, G_{n-1}$ become activated  in the $S_{r+1}$-bootstrap percolation process consecutively. At first,  the edges of  $G_0$ become activated  in $S_{r+1}$-bootstrap percolation process, according to the definition of  $U_0$.  Let   $t\geqslant1$ and assume that the edges of  $G_0,  \ldots, G_{t-1}$ are activated. Since any  vertex  $(v, t)\in V(G_t)$ is  incident to $t$ activated edges with endpoints in $\{(v, i) \, | \, 0\leqslant i\leqslant t-1\}$, we conclude that the edges of $G_t$ become activated   in the $S_{r+1}$-bootstrap percolation process  on $G_t$ by  considering   $U_t$ as the set of initially activated vertices. Hence, $\bigcup_{t=0}^{n-1}U_t$ is a percolating set    of size $\sum_{t=0}^{n-1} m_e(G, r-t)$ in the $S_{r+1}$-bootstrap percolation process on $G\square H$.
\end{proof}

\begin{thm}\label{lbknd}
Let  $n, r, d$  be positive integers with   $n\geqslant r+1$. Then  $m_e(K_n^d, r)={{d+r}\choose{d+1}}$.
\end{thm}

\begin{proof}
First, we prove by induction  on $d$ that there exists a proper edge coloring  $c_d : E(G)\longrightarrow\mathbbmsl{R}$ such that   $\dim W_{c_d}(K_n^d, r)\geqslant{{d+r}\choose{d+1}}$. In view of    Lemma  \ref{lmekn},  there is nothing to prove for  $d=1$. By     Lemma  \ref{car} and the induction  hypothesis, there is a proper edge coloring  $c_d : E(K_n^d)\longrightarrow\mathbbmsl{R}$ such that
\begin{align*}
\dim W_{c_d}\!\left(K_n^d, r\right)&\geqslant\sum_{t=0}^{n-1}\dim W_{c_{d-1}}\!\left(K_n^{d-1}, r-t\right)\\&
\geqslant\sum_{t=0}^{r-1}{{d-1+r-t}\choose{d}}\\&
={{d+r}\choose{d+1}}.
\end{align*}
  It follows from  Theorem \ref{hh} that   $m_e(K_n^d,r)\geqslant{{d+r}\choose{d+1}}$.
Now, we establish  by induction  on $d$ that  $m_e(K_n^d,r)\leqslant{{d+r}\choose{d+1}}$.
The edges of $K_n$  with two endpoints  in $[\![r+1]\!]$    clearly form  a percolating set  in the $S_{r+1}$-bootstrap percolation process on $K_n$ and so there is nothing to prove for  $d=1$.   By applying  Lemma \ref{umeg} and  the induction  hypothesis, we obtain  that
\begin{align*}m_e\!\left(K_n^d, r\right)&\leqslant\sum_{t=0}^{n-1}m_e\!\left(K_n^{d-1}, r-t\right)\\&
\leqslant\sum_{i=0}^{r-1}{{d-1+r-t}\choose{d}}\\&
={{d+r}\choose{d+1}}.
\qedhere\end{align*}
\end{proof}

\section{Multi-dimensional Hamming graphs}

 Balister,    Bollob\'{a}s,    Lee   and   Narayanan   \cite{balister}  gave    the  lower bound  $(r/d)^d$   and the  approximate upper bound  $r^d/(2d!)$ on $m(K_n^d, r)$. In this   section, we improve both bounds  which   result in   an asymptotic formula for $m(K_n^d, r)$.  To begin with,  let us fix the notation we shall use throughout this   section.  We set   $d\geqslant2$ and $\delta=(d-2)/(d-1)$.
For a point  $t=(t_1, \ldots, t_d)\in\{0, 1\}^d$  and a subset $P\subseteq[\![n]\!]^d$, we define
$$P(t)=\left\{(x_1, \ldots, x_d)\in[\![n]\!]^d  \, \left| \,
\begin{array}{ll}
    \text{There exists  $(p_1, \ldots, p_d)\in P$ such that} \\
    \text{$x_i=t_i(n-1-p_i)+(1-t_i)p_i$ for all $i$.}
   \end{array}
    \right.
\right\}.$$
Roughly speaking,  $P(t)$  is a region in  $[\![n]\!]^d$  congruent to $P$ around the point $(n-1)t$ instead of the origin.
For the sets
$$A_r^d=\left\{(x_1,   \ldots, x_d)\in[\![n]\!]^d \, \,  \left| \, \,   \sum_{i=1}^d x_i\leqslant\left\lceil\tfrac{r}{2}\right\rceil-1\right.\right\},$$
$$B_r^d=\left\{(x_1,   \ldots, x_d)\in[\![n]\!]^d \, \, \left| \,  \, x_1+x_2+\delta\sum_{i=3}^d x_i<\delta\left(\left\lceil\tfrac{r}{2}\right\rceil-1\right)\right.\right\}$$
and $C_r^d=A_r^d\setminus B_r^d$, we  define   $$\mathcal{A}_r^d=\bigcup_{t\in T}A_r^d(t), \, \, \mathcal{B}_r^d=\bigcup_{t\in T}B_r^d(t) \,   \text{ and }  \,   \mathcal{C}_r^d=\bigcup_{t\in T}C_r^d(t),$$
where $T=\left\{(t_1, \ldots, t_d)\in \{0,1\}^d \, \Big| \,  t_1=t_2\right\}$.

\begin{lemma}\label{tupor}
Let $n, r, d$ be      positive integers  with $n\geqslant r+1$ and $d\geqslant2$. Then  $\mathcal{A}_r^d$ is a percolating set  of  $K_n^d$  in the $r$-neighbor bootstrap percolation process.
\end{lemma}

\begin{proof}
Let $s=\ceil{r/2}$.
We use an induction argument  on $d$. Theorem \ref{dim2} concludes  the assertion for $d=2$. Let  $d\geqslant3$ and  assume that the assertion holds for $d-1$.
Set   $P_i=\{(x_1, \ldots, x_d)\in[\![n]\!]^d  \,  | \,  x_d=i\}$ and $Q_i=P_i\cap\mathcal{A}_r^d$.
It is not hard to check  that, after  ignoring the last coordinate, both $Q_i$ and $Q_{n-1-i}$ are exactly $\mathcal{A}_{r-2i}^{d-1}$ for any  $i\in[\![s]\!]$.

We consider the following iterative procedure for any  $i\in[\![s]\!]$. At  step $i$, we show that  the vertices in   $P_i\cup P_{n-1-i}$ become activated. The induction hypothesis implies
that  all vertices in $P_0$ and $P_{n-1}$ are activated   by $Q_0$ and $Q_{n-1}$, respectively. Hence,  there is nothing to prove   for $i=0$. Assume that  $i\geqslant1$.  Each vertex
in $P_i\cup P_{n-1-i}$ has already $2i$ activated  neighbors from the previous steps.  So, in order to activate the vertices in $P_i\cup P_{n-1-i}$, it is enough to consider the  $(r-2i)$-neighbor bootstrap percolation process on  $P_i\cup P_{n-1-i}$. This is done  by the induction hypothesis and  by considering      $Q_i\cup Q_{n-1-i}$ as the  initially  activated  set, since  both  $Q_i$ and $Q_{n-1-i}$ are copies of $\mathcal{A}_{r-2i}^{d-1}$.

Finally,  we observe that  any  vertex  in  $\bigcup_{i=s}^{n-s-1}P_i$ has at least $r$  neighbors in  $\bigcup_{i=0}^{s-1}(P_i\cup P_{n-1-i})$ and so it becomes  activated. This completes the proof,
since     $\bigcup_{i=0}^{n-1} P_i=[\![n]\!]^d$  and $\bigcup_{i=0}^{n-1}Q_i=\mathcal{A}_r^d$.
\end{proof}

\begin{lemma}\label{tukhali}
Let $n, r, d$ be      positive integers  with  $n\geqslant r+1$ and $d\geqslant2$. Then 	  $\mathcal{C}_r^d$ is a percolating set of      $K_n^d$  in the $r$-neighbor bootstrap percolation process.
\end{lemma}

\begin{proof}
By   Lemma \ref{tupor}, it suffices to    prove  that all vertices   in $\mathcal{B}_r^d$ become  activated  in
the $r$-neighbor bootstrap percolation process on $K_n^d$.
Note that once a vertex  in $B_r^d$ becomes activated, the corresponding vertices  in all other  $B_r^d(t)$ become   simultaneously activated, due to symmetry.
So, it is sufficient to show  that
any  vertex  in $B_r^d$ becomes activated  in
the $r$-neighbor bootstrap percolation process on $K_n^d$.
Since    $B_r^2=\varnothing$, we may  assume that   $d\geqslant3$.
Fix   an arbitrary  vertex  $x=(x_1,x_2,\ldots,x_d)\in B_r^d$ and denote by
$\eta_{x}^{i}$,  the number  of   neighbors of $x$ in $\mathcal{C}_r^d$ differing from $x$ in the coordinate $i$. Let $\eta_x=\eta_{x}^1+\cdots+\eta_{x}^d$ and $\sigma_x=x_3+\cdots+x_d$.
It straightforwardly  follows from  the definitions  of   $A_r^d$,   $B_r^d$ and  $\mathcal{C}_r^d$  that $\eta_{x}^1=\eta_{x}^2=s-\sigma_x-\ceil{\delta(s-1-\sigma_x)}$  and  $$\eta_{x}^3=\cdots=\eta_{x}^d=2\left(\left\lfloor\frac{x_1+x_2}{d-2}\right\rfloor+1\right),$$
where $s=\ceil{r/2}$. Therefore,
$$\eta_x=2\left(s-\sigma_x-\left\lceil\delta\left(s-1-\sigma_x\right)\right\rceil\right)+2(d-2)\left(\left\lfloor\frac{x_1+x_2}{d-2}\right\rfloor+1\right).$$
Since $s\geqslant r/2$ and  $$\left\lfloor\frac{x_1+x_2}{d-2}\right\rfloor\geqslant\frac{x_1+x_2-(d-3)}{d-2},$$
we obtain    that $\eta_x\geqslant r-2(\rho_x+\sigma_x)$,
where  $\rho_x=\ceil{\delta(s-1-\sigma_x)}-(x_1+x_2+1)$.
Note that $\rho_x\geqslant0$  in view of   the definition of    $B_r^d$.

We now   prove     by   induction  on    $\tau_x=\rho_x+2\sigma_x$ that any vertex $x\in B_r^d$ becomes activated  in
the $r$-neighbor bootstrap percolation process on $K_n^d$.
If    $\tau_x=0$, then $\rho_x=\sigma_x=0$   and it follows from $\eta_x\geqslant r-2(\rho_x+\sigma_x)$ that     $x$  has at least $r$ activated  neighbors, we are done.
So, we may assume  that $\tau_x\geqslant1$.
In view of  the inequality  $\eta_x\geqslant r-2(\rho_x+\sigma_x)$, it is sufficient to show  that at least $2(\rho_x+\sigma_x)$ neighbors of  $x$ in $\mathcal{B}_r^d$   have  been activated during the previous induction steps.
For this, consider the sets
$$P_x=\bigcup_{i=1}^2\left\{w\in[\![n]\!]^d  \, \left| \,
\begin{array}{ll}
    \text{$x$ and $w$  coincide in all components except the} \\
    \text{$i$th component and $w_i\in\{x_i+1, \ldots, x_i+\rho_x\}$.}
   \end{array}
    \right.
\right\},$$
$$Q_x=\bigcup_{i=3}^d\left\{w\in[\![n]\!]^d  \, \left| \,
\begin{array}{ll}
    \text{$x$ and $w$  coincide in all components except} \\
    \text{the $i$th component and    $w_i\in[\![x_i]\!]$.}
   \end{array}
    \right.
\right\}$$
and
$$Q'_x=\bigcup_{i=3}^d\left\{w\in[\![n]\!]^d  \, \left| \,
\begin{array}{ll}
    \text{$x$ and $w$  coincide in all components except} \\
    \text{the $i$th component and   $n-1-w_i\in[\![x_i]\!]$.}
   \end{array}
    \right.
\right\},$$
where $w=(w_1,\ldots,w_d)$.
Clearly, $P_x\cup Q_x\cup Q'_x\subseteq N(x)\cap\mathcal{B}_r^d$. Further, $\tau_w<\tau_x$  for any vertex $w\in P_x\cup Q_x$.  Therefore,   by the induction hypothesis and  the symmetry of   $\mathcal{B}_r^d$, we deduce that  $P_x\cup Q_x\cup Q'_x$ is a   set of activated vertices   of size $2(\rho_x+\sigma_x)$. Thus,  $x$ becomes activated, as required.
\end{proof}

We need the following theorem  in order to prove our result about the upper bound on $m(K_n^d, r)$.

\begin{theorem}[Beged-Dov  \cite{bd}]\label{bdthm}
Let $a_1,\ldots, a_k, b$  be positive   numbers with  $b\geqslant\min\{a_1, \ldots, a_k\}$ and  let  $N$ be the number of  solutions of    $a_1x_1+\cdots+ a_kx_k\leqslant b$ for the  nonnegative integers $x_1, \ldots, x_k$. Then $$\frac{b^k}{k!a_1\cdots a_k}\leqslant N\leqslant\frac{(a_1+\cdots+a_k+b)^k}{k!a_1\cdots a_k}.$$
\end{theorem}

\begin{theorem}
Let $n, r, d$ be      positive integers  with  $n\geqslant r+1$ and $d\geqslant2$.
Then $$\frac{1}{r}{{d+r}\choose{d+1}}\leqslant m\!\left(K_n^d, r\right)\leqslant\frac{(r+2d-1)^d-\delta^2(r-2)^d}{2d!}.$$
\end{theorem}

\begin{proof}
The lower bound is obtained  from Theorem \ref{lbknd} and the fact  that  $m_e(G,r)\leqslant rm(G,r)$.
For the upper bound, note that $\mathcal{C}_r^d$ is a percolating set in the $r$-neighbor bootstrap percolation process on $K_n^d$ by Lemma \ref{tukhali}. It follows from  $B_r^d\subseteq A_r^d$  and  Theorem \ref{bdthm}  that
\begin{align*}
\left|C_r^d\right|&=\left|A_r^d\right|-\left|B_r^d\right|\\&
\leqslant\frac{\left(d+\left\lceil\tfrac{r}{2}\right\rceil-1\right)^d}{d!}-\frac{\left(\delta\!\left(\left\lceil\tfrac{r}{2}\right\rceil-1\right)\right)^d}{d!\delta^{d-2}}\\&
\leqslant\frac{(r+2d-1)^d-\delta^2(r-2)^d}{2^dd!}.
\end{align*}
As $|T|=2^{d-1}$, we have
$$\left|\mathcal{C}_r^d\right|\leqslant\sum_{t\in T}\left|C_r^d(t)\right|\leqslant\frac{(r+2d-1)^d-\delta^2(r-2)^d}{2d!}.$$
This proves the upper bound.
\end{proof}

\begin{cor}
Let $r \rightarrow \infty$, $n\geqslant r+1$ and $d=o(\!\sqrt{r})$. Then
$$\frac{r^d}{(d+1)!}\big(1+o(1)\big)\leqslant m\!\left(K_n^d, r\right)\leqslant\frac{r^d(2d-3)}{2d!(d-1)^2}\big(1+o(1)\big).$$
In particular, if in addition $d\rightarrow \infty$, then $m(K_n^d, r)=\frac{1+o(1)}{(d+1)!}r^d.$
\end{cor}

\section{Line Graphs}

The {\sl line graph} of a graph $G$, written $L(G)$, is the graph whose vertex set is    $E(G)$ and two vertices are adjacent if they share an endpoint. We determined $m(K_n^2,r)$ in Section \ref{sec2}.
One may think of $K_n^2$ as the line graph of $K_{n,n}$, the complete bipartite graph with parts of size $n$. Inspired by this observation, we study $m(L(K_n),r)$, where $L(K_n)$ is the line graph of the complete graph on $n$ vertices. Note that the $r$-neighbor bootstrap percolation on $L(K_n)$ can be viewed
as an {\sl edge percolation process} on $K_n$ and so it is somehow similar to the $S_{r+1}$-bootstrap percolation on $K_n$. In the former, an edge of $K_n$  becomes activated if the number of activated edges incident with  either of its end points is at least $r$ while in the latter, an edge of $K_n$ becomes activated when there are at least $r$ activated edges all incident with one of its end points.

By Theorem \ref{lbknd}, $m_e(K_n,r)=\binom{r+1}{2}$ for $n\geqslant r+1$ which resolves the minimum size of a percolating set in the $S_{r+1}$-bootstrap percolation on $K_n$. In this section, we compute $m(L(K_n),r)$ using our interpretation of the $r$-neighbor bootstrap percolation on $L(K_n)$ as the edge percolation process on $K_n$.
Note that $L(K_n)$ is a $(2n-4)$-regular graph and so if $n\leqslant \lceil{\frac{r}{2}\rceil}+1$, no percolation occurs in $L(K_n)$, implying $m\big (L(K_n),r\big )=\binom{n}{2}$. Hence, the problem is interesting only for $n\geqslant \lceil{\frac{r}{2}\rceil}+2$.

To obtain an upper bound, we
introduce a subset of $E(K_n)$ of size $\lfloor (r+2)^2/8 \rfloor$ whose activation leads to the activation of $E(K_n)$ in the  edge percolation process on $K_n$.
	
\begin{definition}
		Let $r,n$ be nonnegative integers with $n\geqslant \lceil{\frac{r}{2}\rceil}+2$. Define the graph $G^n_{r}$ as follows.
   Let  $[\![n]\!]$ be the vertex set  and
  for $i=0,\ldots,\lceil{r/2}\rceil-1$,  connect $i$ to the last $\lceil{r/2}\rceil-i$ vertices. If $r$ is even, then also connect $n-3+2j-r/2$ to  $n-2+2j-r/2$ for $1\leqslant j \leqslant \lceil{r/4\rceil}$.
	\end{definition}		
		
The condition $n\geqslant \lceil{\frac{r}{2}\rceil}+2$ ensures that  $G^n_{r}$ is a  simple graph with
$$\big|E  (G^n_{r} )\big|=\sum_{i=0}^{\lceil{\frac{r}{2}\rceil}-1}\left(\Big\lceil\frac{r}{2}\Big\rceil-i\right)+\epsilon\sum_{j=1}^{\lceil{\frac {r}{4}\rceil}}1 =\binom{\lceil{\frac{r}{2}\rceil+1}}{2}+\epsilon\Big\lceil \frac r4\Big\rceil= \Big\lfloor{ \frac{(r+2)^2}{8}}\Big\rfloor,$$
where $\epsilon=1$ if $r$ is even and $0$ otherwise.

\begin{lemma}\label{upperline} If $n\geqslant \lceil{\frac{r}{2}\rceil}+2$, then $m(L(K_n),r)\leqslant \big\lfloor{ {(r+2)^2}/{8}}\big\rfloor$.
	\end{lemma}

\begin{proof}
    We show that the activation  of $E\big (G^n_{r}\big )$  leads to the activation of $E(K_n)$ in the  edge percolation process on $K_n$. From the definition, we see that the subgraph of $G^n_{r}$ induced on $[\![n-1]\!]$ is $G^{n-1}_{r-2}$. This proposes to  use an induction argument on $r$. The assertion trivially holds for $r=0,1$. Assume that $r\geqslant 2$. By the definition and  some calculations, one can find that $\deg(n-i)=\max \{ 0, \lfloor \frac r2 \rfloor-i+2 \}$ for $i=2, \ldots, n-\lceil \frac r2 \rceil$. Also, $\deg(n-1)=\lceil \frac r2 \rceil+\epsilon$, where $\epsilon$ is 1 if $r\equiv 2 \pmod{4}$ and 0, otherwise. Note that the set of vertices of $G^n_{r}$ which are not adjacent to $n-1$ is $\{ \lceil \frac r2 \rceil ,\ldots,  n-2-\epsilon \}$. As $\deg(n-1)+\deg(n-2-\epsilon)=\lceil \frac r2 \rceil+\epsilon + \lfloor \frac r2 \rfloor - \epsilon=r$, the edge between $n-1$ and $n-2-\epsilon$ becomes activated and the degree of $n-1$ increases by $1$. Using the same argument, the edge between $n-1$ and $n-3-\epsilon$ becomes activated and so on. Once every edge incident to $n-1$ percolates, we may omit $n-1$ and consider the subgraph of $G^n_{r}$ induced on $[\![n-1]\!]$ which is $G^{n-1}_{r-2}$. As each end point of every edge in this graph is adjacent to $n-1$ through an activated edge, we may consider the edge bootstrap percolation process with parameter  $r-2$ on $K_{n-1}$. The hypothesis of the induction implies that
    the activation  of $E\big (G^{n-1}_{r-2}\big )$  leads to the activation of $E(K_{n-1})$,  completing  the proof.
\end{proof}

We next  find a lower bound on $m(L(K_n),r)$.
\begin{lemma}\label{lowerline} If $n\geqslant \lceil{\frac{r}{2}\rceil}+2$, then $m(L(K_n),r)\geqslant \big\lfloor{ {(r+2)^2}/{8}}\big\rfloor$.
	\end{lemma}

\begin{proof}
 Fix positive integers $r,n$ with  $n\geqslant \lceil{\frac{r}{2}\rceil}+2$. Let $A\subset E(K_n)$ be a minimum size set whose activation leads to the activation of $E(K_n)$ in the  edge percolation process on $K_n$. Lemma \ref{upperline} implies that $|A|\leqslant\lfloor {(r+2)^2}/{8}\rfloor$. Let ${e}=(e_0,e_1,\ldots,e_{t-1})$ be an order in which the edges of  $E(G)\setminus A$ become activated, where $t=\binom{n}{2}-|A|$.
We find a maximal subsequence ${f}=(e_{i_0},e_{i_1},\ldots,e_{i_{k-1}})$ of ${e}$ as follows. Let $e_{i_0}=e_0$. If $e_{i_0},\ldots,e_{i_{j-1}}$ are chosen,  then let $e_{i_{j}}$ be the first edge in ${e}$ after $e_{i_{j-1}}$ which is independent from $e_{i_0},\ldots,e_{i_{j-1}}$.

We show that  $k> \lfloor r/4\rfloor$. To prove it, we find an upper bound on $t$. First note that by the definition of   ${f}$,  every edge in  $e$ is incident with some edge in $f$.
Assume that $e_{i_{j}}=x_jy_j$ for $0\leqslant j\leqslant k-1$. Since  $e_{i_{j}}$ becomes activated after the activation of $e_{{i_j}-1}$,   the vertices  $x_j$ and  $y_j$ are  incident with at least $r$ edges in $A\cup \{e_0,e_1,\dots,e_{{i_j}-1}\}$.  Hence the number of edges in $\{e_{i_j},\ldots,e_{t-1}\}$ with one end point in $\{x_j,y_j\}$  is at most  $2n-3-r$. It follows that $t\leqslant k(2n-3-r)$. On the other hand, $t\geqslant {n  \choose 2}-\lfloor {(r+2)^2}/{8}\rfloor$. An easy calculation shows that $k> \lfloor r/4\rfloor$.

By the definition of   ${f}$, $x_j$ (similarly $y_j$) is incident with at most $2j$ edges of $\{e_0,e_1,\dots,e_{{i_j}-1}\}$. Hence,   the set of edges in $A$ incident with either $x_j$ or $y_j$, say $E_j$, is of the size at least  $r-4j$.  Since the end points of all edges in $f$ are distinct, the sets $E_j$ are pairwise disjoint and therefore
$$|A| \geqslant \sum_{j=0}^{k-1} \big|E_j\big| \geqslant \sum_{j=0}^{\lfloor r/4\rfloor}r-4j= \Big \lfloor \frac{(r+2)^2}{8}\Big\rfloor,$$
as desired.
\end{proof}

Since the upper and lower bounds on  $m(L(K_n),r)$ coincide, we have the following result.
\begin{theorem}\label{line}
	Let $n,r$ be two positive integers. Then
	$$m\big (L(K_n),r\big )=\left\{\begin{array}{ll}
	\mathlarger{\Big\lfloor{ \frac{(r+2)^2}{8}}\Big\rfloor},  &   \mbox{ $n\geqslant \lceil{\frac{r}{2}\rceil}+2$} ;\\
		\vspace{-1mm}\\
		\mathlarger{\binom{n}{2}}, &  \mbox{ o.w.}
		\end{array}\right.$$
\end{theorem}

\section{Concluding remarks}

For $n \geqslant r+1$, as we have seen, $m(K_n^d,r)$ is independent of $n$.  For $n\leqslant r$, it seems that $m(K_n^d,r)$ depends on $n$ and so  in this case it would probably be much harder to derive a formula for $m(K_n^d,r)$. The special case $n=2$ has been asymptotically determined in \cite{ham,mor}.
It is easily checked that $m(K_n^d,1)=1$ and $m(K_n^d,2)=\lceil r/2 \rceil+1$. Using the result $m(K_2^d,3)=\lceil d(d+3)/6\rceil + 1$ of \cite{mor}, one may show that $m(K_n^d,3)\leqslant \lceil (d+1)(d+5)/6\rceil + 1$. On the other hand, by Theorem \ref{lbknd}, $m(K_n^d,3)\geqslant \lceil d(d+5)/6\rceil + 1$. It would be challenging to find $m(K_n^d,3)$ for $n\geqslant 3$. Another  interesting problem is the  determination of  $m_e(L(K_n),r)$ using the polynomial method.

\section*{Acknowledgments}
The second author  is supported in part by  the  Institute for Research in Fundamental Sciences\,(IPM) in 2018/2019.  This paper was written   while he  was visiting  IPM in February  2019.  He
wishes to express his gratitude for the hospitality and support he received from  IPM.

\end{document}